\documentclass[review,3p,11pt,square,sort]{elsarticle}

\usepackage{amssymb}
\usepackage{amsmath}
\usepackage{amsthm}
\usepackage{stmaryrd}
\usepackage{bbm}
\usepackage{enumitem}

\usepackage{fancyhdr}
\usepackage{hyperref}
\hypersetup{%
  pdftitle={BSDEs},%
  pdfsubject={BSDEs},%
  pdfauthor={ShengJun FAN},%
  pdfkeywords={BSDEs},%
  pdfstartview=FitH,%
  CJKbookmarks=true,%
  bookmarksnumbered=true,%
  bookmarksopen=true,%
  colorlinks=true, linkcolor=blue, urlcolor=blue, citecolor=blue, %
}

\usepackage{cleveref}
\crefname{thm}{Theorem}{Theorems}
\crefname{pro}{Proposition}{Propositions}
\crefname{lem}{Lemma}{Lemmas}
\crefname{rmk}{Remark}{Remarks}
\crefname{cor}{Corollary}{Corollaries}
\crefname{dfn}{Definition}{Definitions}
\crefname{ex}{Example}{Examples}
\crefname{section}{Section}{Sections}
\crefname{subsection}{Subsection}{Subsections}

\newcommand{\eps}{\varepsilon}

\newcommand{\To}{\rightarrow}
\newcommand{\as}{{\rm d}\mathbb{P}\times{\rm d} t-a.e.}

\newcommand{\ps}{\mathbb{P}-a.s.}

\newcommand{\F}{\mathcal{F}}
\newcommand{\E}{\mathbb{E}}

\newcommand{\s}{\mathcal{S}}
\newcommand{\M}{{\rm M}}
\newcommand{\hcal}{\mathcal{H}}

\newcommand{\T}{[0,T]}
\newcommand{\LT}{{\mathbb L}^2(\F_T; \R^k)}

\newcommand{\R}{{\mathbb R}}

\newcommand {\Dis}{\displaystyle}

\newtheorem{thm}{Theorem}[section]

\newtheorem{pro}[thm]{Proposition}
\newtheorem{rmk}[thm]{Remark}
\newtheorem{cor}[thm]{Corollary}

\newtheorem{ex}[thm]{Example}

\begin{document}
\begin{frontmatter}

\title{{\boldmath\bf General time interval multidimensional BSDEs with
generators satisfying a weak stochastic-monotonicity\vspace{0.2cm} condition}\tnoteref{found}}
\tnotetext[found]{Fan's research is partially supported by the
the Fundamental Research Funds for the Central Universities (No.\,2017XKZD11).\vspace{0.2cm}}


\author{Tingting Li, Ziheng Xu}
\author{Shengjun Fan\vspace{0.3cm}\corref{cor1}}%
\ead{f\_s\_j@126.com;\ shengjunfan@cumt.edu.cn}
\cortext[cor1]{Corresponding author\vspace{0.1cm}}

\address{School of Mathematics, China University of Mining and Technology, Xuzhou, Jiangsu, 221116, PR China\vspace{-1.0cm}}

\begin{abstract}
This paper establishes an existence and uniqueness result for the adapted solution of a general time interval multidimensional backward stochastic differential equation (BSDE), where the generator $g$ satisfies a weak stochastic-monotonicity condition and a general growth condition in the state variable $y$, and a stochastic-Lipschitz condition in the state variable $z$. This unifies and strengthens some known works. In order to prove this result, we develop some ideas and techniques employed in \citet{XiaoFan2017Stochastics} and \citet{LiuLiFan2019CAM}. In particular, we put forward and prove a stochastic Gronwall-type inequality and a stochastic Bihari-type inequality, which generalize the classical ones and may be useful in other places. The martingale representation theorem, It\^{o}'s formula  and the BMO martingale tool are used to prove these two inequalities.\vspace{0.2cm}
\end{abstract}

\begin{keyword}
Backward stochastic differential equation \sep General time interval \sep \\
\hspace{2cm} Weak stochastic-monotonicity condition \sep Existence and uniqueness\sep\\
\hspace{2cm}  Stochastic Gronwall-type inequality \sep Stochastic Bihari-type inequality \vspace{0.2cm}

\MSC[2010] 60H10\vspace{0.2cm}
\end{keyword}

\end{frontmatter}
\vspace{-0.3cm}

\section{Introduction}
\label{sec:1-Introduction}
\setcounter{equation}{0}

Let us fix two positive integers $k$ and $d$, a finite or infinite terminal time $T$ satisfying $0<T\leq\infty$, and a $d$-dimensional standard Brownian motion $(B_t)_{t\in\T}$ on a completed and filtrated probability space $(\Omega,\F, (\F_t)_{t\in\T}, \mathbb{P})$, where $(\F_t)_{t\in\T}$ is the natural, right-continuous and completed $\sigma$-algebra filtration generated by the $B_\cdot$ We consider the following multidimensional backward stochastic differential equations (BSDEs for short):
\begin{equation}\label{eq:1}
  y_t=\xi+\int_t^T g(s,y_s,z_s){\rm d}s-\int_t^T z_s {\rm d}B_s, \ \ t\in\T,
\end{equation}
where $\xi$ is an $\F_T$-measurable $\R^k$-valued random vector called the terminal condition, the stochastic function
$$g(\omega, t, y, z):\Omega\times\T\times\R^k\times\R^{k\times d}\mapsto\R^k$$
is $(\F_t)$-progressively measurable for each $(y,z)$ called the generator, and the pair of processes $(y_t,z_t)_{t\in\T}\in \R^k\times\R^{k\times d}$ is $(\F_t)$-progressively measurable called the solution of equation \eqref{eq:1}, denoted usually by BSDE$(\xi,T,g)$.\vspace{0.2cm}

It is well known that linear BSDEs were initially introduced by \citet{Bismut1973JMAA} for solving the optimal control problem, and nonlinear BSDEs were first investigated by \citet{PardouxPeng1990SCL}. An existence and uniqueness result for the solution of a finite time interval multidimensional BSDE was at the first time established in \cite{PardouxPeng1990SCL} under a uniform Lipschitz condition of the generator $g$ in the state variables $(y,z)$. Here and hereafter, without special illustration, the word ``uniform" means that the constant in the (Lipschitz) condition is uniform in the two state variables $t$ and $\omega$ of the generator $g$, i.e, the constant is independent of $(t,\omega)$. Furthermore, \citet{Mao1995SPA}, \citet{Pardoux1999NADEC} and \citet{FanJiangDavison2013FMC} respectively weakened the uniform Lipschitz condition in $y$ to a uniform non-Lipschitz condition, a uniform monotonicity condition and a uniform Osgood condition. And, \citet{FanJiang2013AMSE} (see also \citet{Fan2015JMAA}) unified these conditions and established an existence and uniqueness result for the solution of a finite time interval multidimensional BSDE, where the generator $g$ satisfies a uniform weak monotonicity condition with a general growth condition in $y$ and the uniform Lipschitz condition in $z$. Up to now, BSDEs have been successfully applied to many various areas such as PDEs, mathematical finance, optimal control and so on, see, for example, \citet{ElKarouiPengQuenez1997MF}, \citet{Morlais2009FS} and \citet{Peng2004LNM} for details.\vspace{0.2cm}

To the best of our knowledge, \citet{ChenWang2000JAMSA} first investigated infinite time interval BSDEs, put forward a non-uniform (in $t$) Lipschitz condition of the generator $g$ in $(y,z)$, and established the existence and uniqueness for solutions of the BSDEs. Recently, \citet{Morlais2009FS} and \citet{XiaoFan2017Stochastics} respectively relaxed the non-uniform (in $t$) Lipschitz condition of $g$ in $y$ to a non-uniform (in $t$) monotonicity condition and a non-uniform (in $t$) weak monotonicity condition, see also \citet{Fan2016SPA} for more details. Very recently, \citet{LiuLiFan2019CAM} established an existence and uniqueness result for the solution of a general time interval multidimensional BSDE under a non-uniform (in both $t$ and $\omega$) Lipschitz condition of the generator $g$ in $(y,z)$, called the stochastic-Lipschitz condition. The stochastic-Lipschitz condition of $g$ in $y$ was further weakened to a non-uniform (in both $t$ and $\omega$) monotonicity condition (called the stochastic-monotonicity condition) in \citet{LuoFan2018SD} for one-dimensional BSDEs. Readers are refereed to \citet{ElKarouiHuang1997NMIF}, \citet{BenderKohlmann2000}, \citet{WangRanChen2007JAMSA} and \citet{BriandConfortola2008SPA} for another kind of stochastic conditions on the generator $g$, in which some stronger integrability assumptions on the generator and the terminal condition as well as the solutions are required. \vspace{0.2cm}

In this paper, we prove an existence and uniqueness result for the solution of a general time interval multidimensional BSDE, where the generator $g$ satisfies a weak stochastic-monotonicity condition with a general growth condition in $y$, and the stochastic-Lipschitz condition in $z$, see \cref{thm:MainResult} in \cref{sec:4-Existence and uniqueness} for details. Since the assumptions \ref{H:WeakMonotonicityInY} and \ref{H:LipschitzInZ} used in \cref{thm:MainResult} are more general than those in the existing works (see \cref{rmk:4.1} in \cref{sec:4-Existence and uniqueness}), it strengthenes some corresponding works mentioned in the last two paragraphs including Theorem 3.1 in \citet{LiuLiFan2019CAM} and Theorem 6 in \citet{XiaoFan2017Stochastics} for the case of the finite variation process $V_\cdot\equiv 0$, and some new and intrinsic difficulties arise naturally when proving it, see \cref{rmk:4.5} in \cref{sec:4-Existence and uniqueness} for more details.  In order to prove \cref{thm:MainResult}, we put forward and prove a stochastic Gronwall-type inequality and a stochastic Bihari-type inequality by virtue of the martingale representation theorem, It\^{o}'s formula and the BMO (bounded in mean oscillation) martingale tool, see \cref{pro:Gronwall} and \cref{pro:Bihari} in \cref{sec:3-Stochastic inequalities} for more details. These two stochasitc inequalities generalize the classical ones, and may be useful in some other places. Based on these two inequalities and some similar computations employed in \cite{XiaoFan2017Stochastics} and \cite{LiuLiFan2019CAM}, by dividing the time interval $\T$ into a finite number of subintervals with stopping time ends we successfully overcome the arising difficulties in our framework and give the proof of \cref{thm:MainResult}. As a corollary of \cref{thm:MainResult}, we also prove a general existence and uniqueness result for the solution of a multidimensional BSDE with general stopping time interval, see \cref{cor:MainResult} in  \cref{sec:4-Existence and uniqueness} for details. Finally, we give an example (see \cref{ex:L2UnboundedParameters} in  \cref{sec:4-Existence and uniqueness}) in which \cref{thm:MainResult} and \cref{cor:MainResult} can be applied, but any other known results can not be applied.\vspace{0.2cm}

The rest of the paper is organized as follows. In the next section, we introduce some notations which will be used later. The stochastic Gronwall-type and Bihari-type inequalities are stated and proved in \cref{sec:3-Stochastic inequalities}. And, the existence and uniqueness results on the BSDEs are stated and proved in \cref{sec:4-Existence and uniqueness} and \cref{sec:5-proof of existence and uniqueness result}, respectively.

\section{Notations}
\label{sec:2-Notations}

In this section, we introduce some notations used later. First, denote the interval $[0,+\infty)$ by $\R_+$, the Euclidean norm of $y\in \R^{n}$ by $|y|$ for each $n\geq 1$, and the indicator function of $A$ by ${\bf 1}_A$ for each set $A$. Then, let $\LT$ be the set of all $\F_T$-measurable $\R^k$-valued random vectors $\xi$ satisfying $\E[|\xi|^2]<+\infty$, $\s^2(0,T; \R^k)$ the set of all $(\F_t)$-progressively measurable and continuous $\R^k$-valued processes $(Y_t)_{t\in\T}$ such that
$$
\|Y\|_{{\s}^2}:=\left(\E\left[\sup_{t\in\T} |Y_t|^2\right]\right)^{1/2}<+\infty,\vspace{0.2cm}
$$
and $\M^2(0,T; \R^{k\times d})$ the set of all $(\F_t)$-progressively measurable $\R^{k\times d}$-valued processes $(Z_t)_{t\in\T}$ satisfying
$$
\|Z\|_{\M^2}:=\left\{ \E\left[\int_0^T |Z_t|^2{\rm d}t\right] \right\}^{1/2}<+\infty,\vspace{0.2cm}
$$
and $\hcal^2(0,T; \R^k)$ the set of all $(\F_t)$-progressively measurable $\R^k$-valued processes $(X_t)_{t\in\T}$ such that
$$
\|X\|_{\hcal^2}:=\left\{\E\left[\left(\int_0^T |X_t|{\rm d}t\right)^2\right]\right\}^{1/2}<+\infty.\vspace{0.3cm}
$$

Furthermore, for a local $\R^k$-valued or real-valued martingale $\int_0^\cdot {z_s {\rm d}B_s}$, we say that it is a martingale of bounded mean oscillation (${\rm BMO}$-martingale in short) means that\vspace{0.2cm}
$$
\sup\limits_{\tau\in\Sigma_T}\bigg\| \E\bigg[\int_\tau^T{|z_s|^2{\rm d}s}\bigg|\F_\tau\bigg]\bigg\|_\infty
<+\infty,\vspace{0.2cm}
$$
here and hereafter, $\Sigma_T$ denotes the set of all $(\F_t)$-stopping times $\tau$ valued in $\T$, and $\|\xi\|_\infty$ the infinity norm of the essentially bounded real-valued random variable $\xi$, i.e.,
$$
\|\xi\|_\infty:=\sup\{x\in\R_+:\mathbb{P}(|\xi|>x)>0\}.
$$
And, we use $L^\infty(\Omega; L^1(\T;\R_+))$ and $L^\infty(\Omega; L^2(\T;\R_+))$ to denote, respectively, the set of all $(\F_t)$-progressively measurable nonnegative real-valued processes 
$$
u_t(\omega),\ v_t(\omega):\Omega\times[0, T]\mapsto\R_+\vspace{-0.2cm}
$$ 
satisfying
$$
  \bigg\|\int_0^T{u_s(\omega) {\rm d}s}\bigg\|_\infty<+\infty\ \ {\rm and}\ \ \bigg\|\int_0^T{v^2_s(\omega) {\rm d}s}\bigg\|_\infty<+\infty.\vspace{0.4cm}
$$

Finally, denote by ${\bf S}$ the set of non-decreasing continuous functions $\rho(x):\R_+\To \R_+$ satisfying the following conditions:
\begin{itemize}
\item [(i)] $\rho(0)=0$, $\rho(x)>0$ for each $x>0$, and $\int_{0^+}{{\rm d}x\over \rho(x)}:=\lim\limits_{\epsilon\To 0^+}\int_0^\epsilon {{\rm d}x\over \rho(x)}=+\infty$;
\item [(ii)] There exists a constant $A\geq 1$ such that $0\leq \rho(x)\leq A(1+x)$ for each $x\geq 0$;
\item [(iii)] For each real $c>0$, the derivative function of $\rho$ on interval $[c,+\infty)$ is locally bounded, i.e., there exists a constant $M_c>0$ depending only on $c$ such that $0\leq \rho'(x)\leq M_c$ for each $x\in [c,+\infty)$.
\end{itemize}
We remark that if $\rho(x):\R_+\To \R_+$ is a non-decreasing, concave and derivative function satisfying condition (i), then $\rho(\cdot)\in {\bf S}$ since conditions (ii) and (iii) hold automatically in this case. However, we also remark that a function 
$\rho(\cdot)$ belonging to ${\bf S}$ is not necessarily concave.

\section{Stochastic Gronwall-type and Bihari-type inequalities}
\label{sec:3-Stochastic inequalities}

In this section, we will put forward and prove a stochastic Gronwall-type inequality and a stochastic Bihari-type inequality, which respectively generalize the classical ones. Classical proof methods seem to be not applicable in the stochastic framework, and it is interesting that the martingale theorem, It\^{o}'s formula and the BMO martingale tool play a crucial role in our proof of these two inequalities. These two inequalities will be employed in \cref{sec:5-proof of existence and uniqueness result} to prove the existence and uniqueness result for the solution of BSDEs with generators satisfying a weak stochastic-monotonicity condition. We believe that they would be useful in some other places. \vspace{0.2cm}

The following proposition extends the classical Gronwall inequality to a stochastic version. It also generalizes Theorem 1 in \citet{WangFan2019JIA}, which states that in the case of $f_\cdot\equiv 0$ and $\eta\equiv c$ for a constant $c\geq 0$, if \eqref{eq:2} is satisfied for each $t\in\T$, then \eqref{eq:3} holds for each $t\in\T$.

\begin{pro}\label{pro:Gronwall} (Stochastic Gronwall-type inequality)
Assume that $0<T\leq\infty$, $\mu_\cdot$ is an $(\F_t)$-progressively measurable, continuous and nonnegative real-valued process satisfying
$$\E\left[\sup_{t\in[0,T]}\mu_t(\omega)\right]<+\infty,$$
$\eta$ is an $\F_T$-measurable nonnegative real-valued random variable satisfying $\E[\eta]<+\infty$,  $\beta_\cdot$ is a process belonging to $L^\infty(\Omega; L^1(\T;\R_+))$, and both $f_\cdot$ and $h_\cdot$ are $(\F_t)$-progressively measurable nonnegative real-valued processes satisfying
$$
\E\left[\int_0^T[f_t(\omega)+h_t(\omega)] {\rm d}t\right]<+\infty.\vspace{0.2cm}
$$

If for each $t\in\T$, it holds that
\begin{equation}\label{eq:2}
\mu_t \leq \E\left[\eta+\int_t^T(\beta_s\mu_s+f_s){\rm d}s\bigg|\F_t\right],\ \ \ps,
\end{equation}
then for each $t\in\T$, we have
\begin{equation}\label{eq:3}
\mu_t \leq {\rm e}^{\left\|\int_t^T\beta_s {\rm d}s\right\|_\infty} \E\left[\eta+ \int_t^T f_s {\rm d}s \bigg|\F_t\right],\ \ \ps\vspace{0.2cm}
\end{equation}

Moreover, if for each $t\in \T$, it holds that
\begin{equation}\label{eq:4}
\E\left[\sup\limits_{s\in [t,T]}\mu_s+\int_t^Th_s{\rm d}s\bigg|\F_t\right] \leq \E\left[\eta+\int_t^T(\beta_s\mu_s+f_s){\rm d}s\bigg|\F_t\right],\ \ \ps,
\end{equation}
then for each $t\in \T$, we have
\begin{equation}\label{eq:5}
\mu_t\leq \E\left[\sup\limits_{s\in [t,T]}\mu_s+\int_t^Th_s{\rm d}s\bigg|\F_t\right] \leq {\rm e}^{\left\|\int_t^T\beta_s {\rm d}s\right\|_\infty} \E\left[\eta+ \int_t^T f_s {\rm d}s \bigg|\F_t\right],\ \ \ps\vspace{0.3cm}
\end{equation}
\end{pro}

\begin{rmk}\label{rmk:3.2}
For simplicity of notations, here and hereafter the random processes $\mu_t(\omega)$, $\beta_t(\omega)$ and $f_t(\omega)$ are sometimes abbreviated as $\mu_t$, $\beta_t$ and $f_t$ respectively, and the $\ps$ is usually omitted without causing confusion. We will also adopt similar notations for other processes.
\end{rmk}

\begin{proof}[\bf Proof of \cref{pro:Gronwall}] We will develop the martingale representation method employed in \citet{WangFan2019JIA} to prove this proposition. Set
$$\bar \eta:=\eta+\int_0^T(\beta_s\mu_s+f_s){\rm d}s.$$
From the assumptions of $\eta$, $\beta_\cdot$, $\mu_\cdot$ and $f_\cdot$, it is clear that $\E[\bar\eta]<+\infty$. Then, by the martingale representation theorem (see Theorem 2.46 in \citet{PardouxRasscanu2014book}), there exists an $(\F_t)$-progressively measurable $\R^{1\times d}$-valued process $(z_t)_{t\in [0,T]}$ such that
\begin{equation}\label{eq:6}
\E[\bar\eta |\F_t]=\E[\bar\eta]+\int_0^t z_s {\rm d}B_s,\ \ t\in [0,T].\vspace{-0.2cm}
\end{equation}
Now, let
$$
\bar\mu_t:=\E\left[\eta+\int_t^T(\beta_s\mu_s+f_s) {\rm d}s\bigg|\F_t\right]=
\E\left[\bar\eta\big|\F_t\right]-\int_0^t(\beta_s\mu_s+f_s) {\rm d}s,\ \ t\in [0,T].\vspace{0.2cm}
$$
Then, $\bar\mu_\cdot$ is $(\F_t)$-progressively measurable and, in view of \eqref{eq:6},
$$
\bar\mu_t=\E[\bar\eta]-\int_0^t(\beta_s\mu_s+f_s){\rm d}s+\int_0^t z_s {\rm d}B_s,\ \ t\in[0,T].\vspace{0.1cm}
$$
It follows from It\^{o}'s formula together with the fact of $\mu_\cdot\leq \bar\mu_\cdot$ due to \eqref{eq:2} that\vspace{0.1cm}
\begin{equation}\label{eq:7}
\begin{array}{lll}
\Dis {\rm d}\left(\bar\mu_r{{\rm e}^{\int_0^r{\beta_s{\rm d}s}}}\right)&=& \Dis {\rm e}^{\int_0^r{\beta_s{\rm d}s}}\left[-(\beta_r\mu_r+f_r){\rm d}r+z_r {\rm d}B_r+\beta_r\bar\mu_r{\rm d}r\right]\vspace{0.2cm}\\
&\geq & \Dis {\rm e}^{\int_0^r{\beta_s{\rm d}s}}(-f_r{\rm d}r+z_r {\rm d}B_r),\ \ \ r\in\T.
\end{array}
\end{equation}
Note that the process
$$
\left(\int_0^t e^{\int_0^r\beta_s{\rm d}s} z_r {\rm d}B_r\right)_{t\in[0,T]}\vspace{0.2cm}
$$
is an $(\F_t)$-martingale. Integrating on the interval $[t,T]$ and taking the conditional expectation with respect to $\F_{t}$ on both sides of \eqref{eq:7}, we obtain that
$$
\E\left[\eta {\rm e}^{\int_0^T{\beta_s{\rm d}s}}-\bar\mu_t{{\rm e}^{\int_0^t{\beta_s{\rm d}s}}}\bigg|\F_t\right]\geq -\E\left[\int_t^T{\rm e}^{\int_0^r\beta_s{\rm d}s}f_r{\rm d}r\bigg|\F_t\right],\ \ t\in\T.
$$
Consequently, in view of the fact that $\bar\mu_t {\rm e}^{\int_0^t\beta_s{\rm d}s}$ is $\F_t$-measurable,
\begin{equation}\label{eq:8}
\bar\mu_t{\rm e}^{\int_0^t\beta_s{\rm d}s}\leq \E\left[\eta {\rm e}^{\int_0^T\beta_s{\rm d}s}+\int_t^T{\rm e}^{\int_0^r\beta_s{\rm d}s}f_r{\rm d}r\bigg|\F_t\right],\ \ t\in\T.
\end{equation}
Then, since $\mu_\cdot\leq \bar\mu_\cdot$, the desired inequality \eqref{eq:3} follows immediately from \eqref{eq:8}. \vspace{0.2cm}

Moreover, if \eqref{eq:4} is satisfied for each $t\in \T$, then it is clear that \eqref{eq:2} holds for each $t\in \T$. So, \eqref{eq:8} also holds. Thus, the desired inequality \eqref{eq:5} follows from inequalities \eqref{eq:8} and \eqref{eq:4} together with the definition of $\bar\mu_\cdot$. The proof is then complete.\vspace{0.2cm}
\end{proof}

The following \cref{pro:Bihari} generalizes the classical Bihari inequality to a stochastic version. We would like to mention that another stochastic Bihari-type inequality was established in Proposition 4.6 of \citet{DingWu1998SPA}, but it has a different form and a different proof from ours. In particular, Proposition 4.6 in \cite{DingWu1998SPA} can not be employed to prove our existence and uniqueness result on the BSDEs in \cref{sec:5-proof of existence and uniqueness result}, and \cref{pro:Bihari} also can not be derived from it.

\begin{pro}\label{pro:Bihari} (Stochastic Bihari-type inequality )
Assume that $c>0$, $0<T\leq\infty$, $\beta_\cdot\in L^\infty(\Omega; L^1(\T;\R_+))$ and $\rho(\cdot)\in {\bf S}$. Let $\mu_\cdot$ be an $(\F_t)$-progressively measurable, continuous and nonnegative real-valued process satisfying
$$
\E\left[\sup\limits_{t\in[0,T]}\mu_t(\omega)\right]<+\infty.\vspace{0.2cm}
$$

If for each $t\in\T$, it holds that
\begin{equation}\label{eq:9}
\mu_t \leq c+\E\left[\int_t^T\beta_s\rho(\mu_s){\rm d}s\bigg|\F_t\right],\ \ \ps,
\end{equation}
then for each $t\in\T$, we have
\begin{equation}\label{eq:10}
\mu_t \leq \Theta^{-1}\left(\Theta(c)+\left\|\int_t^T\beta_s {\rm d}s\right\|_\infty\right),\ \ \ps,
\end{equation}
where
$$
\Theta(x):=\int_1^x {1\over \rho(u)}\ {\rm d}u, \ \ x>0,\vspace{0.2cm}
$$
is a strictly increasing function valued in $\R$, and $\Theta^{-1}$ is the inverse function of $\Theta$.\vspace{0.2cm}

Moreover, if \eqref{eq:9} holds for $c=0$, then $\mu_t\equiv 0$ for each $t\in \T$.\vspace{0.2cm}
\end{pro}

\begin{proof}
Note first that $\rho(x)\leq A(1+x)$ for each $x\in \R_+$ and that $\beta_\cdot\in L^\infty(\Omega; L^1(\T;\R_+))$. It follows from \cref{pro:Gronwall} that for each $t\in\T$,
$$
0\leq \mu_t\leq {\rm e}^{A\left\|\int_0^T\beta_s {\rm d}s\right\|_\infty}\left(c+A\left\|\int_0^T\beta_s {\rm d}s\right\|_\infty\right)=:C.
$$
Then, letting $\eta:=\int_0^T \beta_s\rho(\mu_s){\rm d}s$, we have
\begin{equation}\label{eq:11}
0\leq \eta\leq A(1+C)\left\|\int_0^T\beta_s {\rm d}s\right\|_\infty=:M.
\end{equation}
On the other hand, by the classical martingale representation theorem, there exists an $(\F_t)$-progressively measurable and square-integrable $\R^{1\times d}$-valued process $(z_t)_{t\in [0,T]}$ such that
\begin{equation}\label{eq:12}
\E[\eta |\F_t]=\E[\eta]+\int_0^t z_s {\rm d}B_s,\ \ t\in [0,T].
\end{equation}
From It\^{o}'s formula together with \eqref{eq:12} and \eqref{eq:11}, we deduce that
\begin{equation}\label{eq:13}
\sup\limits_{\tau\in\Sigma_T}\bigg\| \E\bigg[\int_\tau^T{|z_s|^2{\rm d}s}\bigg|\F_\tau\bigg]\bigg\|_\infty= \sup\limits_{\tau\in\Sigma_T}\bigg\| \E\bigg[\eta^2-\left(\E[\eta\big|\F_\tau]\right)^2\bigg|\F_\tau\bigg]\bigg\|_\infty
\leq M^2<+\infty,\vspace{0.1cm}
\end{equation}
which means that $\int_0^\cdot z_s\cdot {\rm d}B_s$ is a BMO-martingale.\vspace{0.2cm}

Next, set
$$
\bar\mu_t:=c+\E\left[\int_t^T \beta_s\rho(\mu_s) {\rm d}s\bigg|\F_t\right]=c+
\E\left[\eta\big|\F_t\right]-\int_0^t \beta_s\rho(\mu_s) {\rm d}s,\ \ t\in [0,T].
$$
Then, $\bar\mu_\cdot$ is $(\F_t)$-progressively measurable, $c\leq \bar\mu_\cdot\leq c+M$ due to \eqref{eq:11}, and in view of \eqref{eq:12},
$$
\bar\mu_t=c+\E[\eta]-\int_0^t \beta_s\rho(\mu_s) {\rm d}s+\int_0^t z_s {\rm d}B_s,\ \ t\in[0,T].\vspace{0.1cm}
$$
It follows from It\^{o}'s formula that, in view of the monotonicity of $\rho(\cdot)$ together with the facts that $\bar\mu_\cdot\geq c$ and $\mu_\cdot\leq \bar\mu_\cdot$ due to \eqref{eq:9},
\begin{equation}\label{eq:14}
\begin{array}{lll}
\Dis {\rm d}\Theta(\bar \mu_s) &=& \Dis \frac{1}{\rho(\bar \mu_s)}\left[-\beta_s\rho(\mu_s){\rm d}s+z_s {\rm d}B_s\right]-\frac{1}{2}\frac{\rho'(\bar \mu_s)}{\rho^2(\bar \mu_s)}|z_s|^2{\rm d}s\vspace{0.2cm}\\
&\geq& \Dis -\beta_s{\rm d}s+\frac{1}{\rho(\bar \mu_s)}z_s\left[{\rm d}B_s-\frac{1}{2}\frac{\rho'(\bar \mu_s)}{\rho(\bar \mu_s)}z_s^* {\rm d}s\right],\ \ s\in\T,
\end{array}
\end{equation}
where and hereafter $z^*_\cdot$ denotes the transpose of the matrix $z_\cdot$.\vspace{0.2cm}

Furthermore, in view of the assumptions of $\rho(\cdot)$ and the fact that $c\leq \bar\mu_\cdot\leq c+M$, we know that $0\leq \rho'(\bar \mu_t)/\rho(\bar \mu_t) \leq M_c/\rho(c)$ for each $t\in \T$. It then follows from \eqref{eq:13} that the process
$$H_t:=\frac{1}{2}\int_0^t \frac{\rho'(\bar \mu_s)}{\rho(\bar \mu_s)}z_s {\rm d}B_s,\ \ t\in \T$$
is a BMO-martingale under probability measure $\mathbb{P}$. Then by Theorem 2.3 in \citet{Kazamaki1994book}, the stochastic exponential of $H_\cdot$,
$${\mathcal E}(H)_t=\exp\left(H_t-{1\over 2}\langle H\rangle_t\right), \ \ t\in\T$$
with $\langle H\rangle_\cdot$ being the quadratic variation process of $H_\cdot$,  is a uniformly integrable martingale under $\mathbb{P}$, and then we can denote a probability measure $\mathbb{Q}$ on $(\Omega,\F_T)$ by ${{\rm d}\mathbb{Q}\over {\rm d}\mathbb{P}}:={\mathcal E}(H)_T$. Thus, note that, in view of \eqref{eq:13} and the fact of $0\leq 1/\rho(\bar \mu_t)\leq 1/\rho(c)$ for each $t\in\T$,
$$\overline {H}_t:=\int_0^t \frac{1}{\rho(\bar \mu_s)}z_s {\rm d}B_s,\ \ t\in \T\vspace{0.1cm}$$
is a also BMO-martingale under $\mathbb{P}$. It follows that the Girsanov's transform of $\overline {H}$,
$$\int_0^t \frac{1}{\rho(\bar \mu_s)}z_s \left[{\rm d}B_s-\frac{1}{2}\frac{\rho'(\bar \mu_s)}{\rho(\bar \mu_s)}z_s^* {\rm d}s\right],\ \ t\in \T,$$
is a BMO-martingale under the probability measure $\mathbb{Q}$.\vspace{0.2cm}

In the sequel, integrating on the interval $[t,T]$ and taking the conditional expectation with respect to $\F_t$ under $\mathbb{Q}$ on both sides of \eqref{eq:14}, we obtain that for each $t\in\T$,
$$
\Theta(c)-\Theta(\bar\mu_t)=\E_\mathbb{Q}\left[\Theta(c)-\Theta(\bar\mu_t)\bigg|\F_t\right]\geq -\E_\mathbb{Q}\left[\int_t^T \beta_s{\rm d}s\bigg|\F_t\right]\geq -\left\|\int_t^T \beta_s{\rm d}s\right\|_{\infty},\ \ \ps,
$$
where $\E_\mathbb{Q}[X|\F_t]$ denotes the conditional expectation of random variable $X$ with respect to $\F_t$ under $\mathbb{Q}$. Then, in view of \eqref{eq:9} and the definition of $\bar\mu_\cdot$, the desired inequality \eqref{eq:10} follows immediately from the last inequality and the strictly monotonicity of the function $\Theta(\cdot)$.\vspace{0.2cm}

Finally, if \eqref{eq:9} holds for $c=0$, then for each $n\geq 1$, we have
$$
0\leq \mu_t \leq \Theta^{-1}\left(\Theta({1\over n})+\left\|\int_0^T\beta_s {\rm d}s\right\|_\infty\right).
$$
The last desired assertion follows by sending $n\To\infty$ in the previous inequality.\vspace{0.2cm}
\end{proof}

\begin{rmk}\label{rmk:3.4}
Let $\overline{\mathbb{P}}$ be a equivalent probability measure to $\mathbb{P}$ defined the space $(\Omega,\F_T)$. From the above proof, it is not difficult to verify that the conclusions in \cref{pro:Bihari} still hold if the expectation and the conditional expection appearing in the assumptions are taken under $\overline{\mathbb{P}}$ rather than $\mathbb{P}$. Consequently, \cref{pro:Bihari} can be compared with Lemma 2.1 in \citet{Fan2016SPA}, where the function $\rho(\cdot)$ does not need to satisfy the condition (iii) in the definition of set ${\bf S}$, but the process $\beta_\cdot$ has to be deterministic, namely, it is independent of the variable $\omega$. In addition, we also mention that due to the randomness of $\beta_\cdot$, the ODE method used to prove Lemma 2.1 in \citet{Fan2016SPA} can not applied to prove \cref{pro:Bihari}.
\end{rmk}

\section{Statement of the existence and uniqueness result}
\label{sec:4-Existence and uniqueness}

Before stating the existence and uniqueness result, let us first introduce the following assumptions on the generator $g$:

\begin{enumerate}
\renewcommand{\theenumi}{(H\arabic{enumi})}
\renewcommand{\labelenumi}{\theenumi}
  \item\label{H:ContinuousInY} $\as$, $g(\omega,t,\cdot,z)$ is continuous for each $z\in\R^{k\times d}$.

  \item\label{H:WeakMonotonicityInY} $g$ satisfies a weak stochastic-monotonicity condition in $y$, i.e., there exists a function $\rho(\cdot)\in {\bf S}$ and a process $u_\cdot\in L^\infty(\Omega; L^1(\T;\R_+))$ such that $\as$, for each $y_1, y_2\in\R^k$ and $z\in\R^{k\times d}$, we have
        $$\langle y_1-y_2,g(\omega, t, y_1, z)-g(\omega, t, y_2, z)\rangle
          \leq u_t(\omega)\rho(|y_1-y_2|^2).$$

\item\label{H:GeneralizedGeneralGrowthInY}
$g$ has a general growth in $y$, i.e., for each $r\in\R_+$, it holds that 
$$\E\left[\int_0^T\psi_{r}(\omega,t){\rm d}t\right]<+\infty$$ 
with 
$$
\psi_{r}(\omega,t):=\sup_{|y|\leq r}|g(\omega,t,y,0)-g(\omega,t,0,0)|.\vspace{0.2cm}
$$ 
And, $g(\omega, t,0,0)\in \hcal^2(0,T;\R^k)$.

  \item\label{H:LipschitzInZ} $g$ satisfies a stochastic-Lipschitz condition in $z$, i.e., there exists a $v_\cdot\in L^\infty(\Omega; L^2(\T;\R_+))$ such that $\as$, for each $y\in\R^k$ and $z_1, z_2\in\R^{k\times d}$,
$$
|g(\omega,t,y,z_1)-g(\omega,t,y,z_2)|\leq v_t(\omega)|z_1-z_2|.
\vspace{0.2cm}
$$
\end{enumerate}

\begin{rmk}\label{rmk:4.1}
Assumption \ref{H:WeakMonotonicityInY} is strictly weaker than both the non-uniform (in $t$) weak monotonicity condition (i.e., $\mu_\cdot$ is independent of $\omega$) and the stochastic-monotonicity condition (i.e., $\rho(x)=x$) of the generator $g$ in $y$ employed respectively in \citet{Fan2016SPA}, \citet{XiaoFan2017Stochastics} and \citet{LuoFan2018SD}. And, assumption \ref{H:LipschitzInZ} is strictly weaker than the non-uniform (in $t$) Lipschitz condition of $g$ in $z$ used in \citet{ChenWang2000JAMSA}, \citet{Morlais2009FS}, \citet{Fan2016SPA} and \citet{XiaoFan2017Stochastics}. \vspace{0.2cm}
\end{rmk}

The following existence and uniqueness theorem is one of the main results of this paper.

\begin{thm}\label{thm:MainResult}
Let $0<T\leq \infty$ and $g$ satisfy assumptions \ref{H:ContinuousInY}--\ref{H:LipschitzInZ}. Then, for each $\xi\in\LT$, BSDE$(\xi,T,g)$ admits a unique solution $(y_t,z_t)_{t\in\T}\in \s^2(0,T;\R^k)\times \M^2(0,T;\R^{k\times d})$.\vspace{0.2cm}
\end{thm}

\begin{rmk}\label{rmk:4.3}
In view of \cref{rmk:4.1}, \cref{thm:MainResult} strengthenes Theorem 3.1 in \citet{LiuLiFan2019CAM} and Theorem 6 in \citet{XiaoFan2017Stochastics} for the case of the finite variation process $V_\cdot\equiv 0$ together with some corresponding existence and uniqueness results obtained, for example, in \citet{PardouxPeng1990SCL}, \citet{Mao1995SPA}, \citet{ChenWang2000JAMSA}, \citet{BriandDelyonHu2003SPA} and \citet{FanJiangDavison2013FMC}.\vspace{0.2cm}
\end{rmk}

The following corollary follows from \cref{thm:MainResult}, which gives a general existence and uniqueness result for the solution of a multidimensional BSDEs with general stopping time interval.

\begin{cor}\label{cor:MainResult}
Let $0<T\leq \infty$ and $\tau$ be any $(\F_t)$-stopping time valued in $\T$. If the generator $g$ satisfies assumptions \ref{H:ContinuousInY}--\ref{H:LipschitzInZ}, then for each $\F_\tau$-measurable $\R^k$-valued random vector $\xi$ satisfying  $\E[|\xi|^2]<+\infty$, BSDE$(\xi,\tau,g)$ admits a unique solution $(y_t,z_t)_{t\in\T}$ in the space $\s^2(0,T;\R^k)\times \M^2(0,T;\R^{k\times d})$ in the sense that $(y_t,z_t)_{t\in\T}\in \s^2(0,T;\R^k)\times \M^2(0,T;\R^{k\times d})$, $\as$, $z_t{\bf 1}_{t\geq \tau}=0$ and $\ps$, the following equation holds:
$$
 y_t=\xi+\int_{t\wedge\tau}^{\tau} g(s,y_s,z_s){\rm d}s-\int_{t\wedge\tau}^{\tau} z_s {\rm d}B_s, \ \ t\in\T.\vspace{0.3cm}
$$
\end{cor}

\begin{rmk}\label{rmk:4.5}
Since the assumptions used in \cref{thm:MainResult} and \cref{cor:MainResult} is more general than those in the existing works, some new and intrinsic difficulties arise naturally when proving them. In particular, due to the presence of the function $\rho(\cdot)$ in assumption \ref{H:WeakMonotonicityInY}, it seems to be impossible to obtain a contraction by virtue of  the weighted norms employed in \citet{ElKarouiHuang1997NMIF}, \citet{BenderKohlmann2000}, and \citet{WangRanChen2007JAMSA}. And, due to the randomness of the processes $\mu_\cdot$ and $\nu_\cdot$ in assumptions \ref{H:WeakMonotonicityInY} and \ref{H:LipschitzInZ}, it seems to be also impossible to obtain a contraction by slicing the whole time interval $\T$ in a finite number of deterministic subintervals, which is employed in \citet{ChenWang2000JAMSA} and \citet{XiaoFan2017Stochastics}. In addition, also due to the randomness of the processes $\mu_\cdot$ and $\nu_\cdot$, the usual (deterministic) Gronwall inequality and Bihari inequality which play a crucial role in \citet{XiaoFan2017Stochastics} are not applicable any longer, and then we have to extend them to a stochastic version.\vspace{0.1cm}
\end{rmk}

The following example shows that \cref{thm:MainResult} and \cref{cor:MainResult} are not covered by any known results.

\begin{ex}\label{ex:L2UnboundedParameters}
Let $k=2$, $0<T\leq \infty$ and $M>0$. Define the following $(\F_t)$-stopping times
$$
\tau _1(\omega):=\inf\left\{t\in \T:\int_0^t |B_s(\omega)|{\rm d}s\geq {M\over 2}\right\}\wedge T
$$
and
$$
\tau _2(\omega):=\inf\left\{t\in \T:\int_0^t |B_s(\omega)|^2{\rm d}s\geq {M\over 2}\right\}\wedge T\vspace{0.2cm}
$$
with the convention $\inf\emptyset=+\infty$, and define the following processes
$$
\bar u_t(\omega):=|B_t(\omega)|{\bf 1}_{t\leq \tau _1(\omega)},\ \  \ {\rm and}\ \ \ \bar v_t(\omega):=|B_t(\omega)|{\bf 1}_{t\leq \tau _2(\omega)},\ \ (\omega, t)\in\Omega\times\T.
$$
It is clear that $\bar u_\cdot\in L^\infty(\Omega; L^1(\T;\R_+))$ and $\bar v_\cdot\in L^\infty(\Omega; L^2(\T;\R_+))$.\vspace{0.2cm}

For $i=1,2$, let $y_i$ and $z_i$ represent, respectively, the $i$th component of the vector $y$ and the $i$th row of the matrix $z$. Consider the following \vspace{0.2cm}generator: for $(\omega, t, y, z)\in\Omega\times\T\times\R^k\times\R^{k\times d}$,
$$
g(\omega,t,y,z)=\bar u_t(\omega)
      \begin{bmatrix}
        \displaystyle h(|y_2|)-{\rm e}^{y_1}\\
        \displaystyle h(|y_1|)-{\rm e}^{y_2}
      \end{bmatrix}
+\bar v_t(\omega)
      \begin{bmatrix}
        \displaystyle |z_2|\\
        \displaystyle |z_1|
      \end{bmatrix}
+     \begin{bmatrix}
        \displaystyle |B_t(\omega)|\\
        \displaystyle |B_t(\omega)|
      \end{bmatrix}
      ,\vspace{0.2cm}
$$
where $h(x):=(-x\ln x){\bf 1}_{0\leq x\leq\delta}+\big(h'(\delta-)(x-\delta)+h(\delta)\big){\bf 1}_{x>\delta}
$ with $\delta$ small enough.\vspace{0.2cm}

It is not very hard to verify that $g$ satisfies assumptions \ref{H:ContinuousInY}--\ref{H:LipschitzInZ} with $u_t(\omega)=\bar u_t(\omega)$, $v_t(\omega)=\bar v_t(\omega)$ and $\rho(x)=h(x)$. It then follows from \cref{thm:MainResult} that for each $\xi\in\LT$, BSDE$(\xi,T,g)$ admits a unique solution in the space $\s^2(0,T;\R^k)\times \M^2(0,T;\R^{k\times d})$. Furthermore, for each $(\F_t)$-stopping time $\tau$ valued in $\T$ and each $\F_\tau$-measurable $\R^k$-valued random vector $\eta$ satisfying $\E[|\eta|^2]<+\infty$, it follows from \cref{cor:MainResult} that BSDE$(\eta,\tau,g)$ admits also a unique solution in the space $\s^2(0,T;\R^k)\times \M^2(0,T;\R^{k\times d})$.\vspace{0.2cm}

We especially mention that the above conclusions can not be obtained by any existing results since this generator $g$ fails to fulfil their assumptions due to the facts that ${\rm e}^x$ has a general growth in $x$, $h(\cdot)$ is not a linear function, and $\bar u_t(\omega)$ and $\bar v_t(\omega)$ can not be, respectively, dominated by two deterministic nonnegative functions $\tilde u_t$ and $\tilde v_t$ defined on $\T$ satisfying
\begin{equation}\label{eq:15}
\int_0^T \tilde u_t {\rm d}t<+\infty\ \ \ {\rm and}\ \ \ \int_0^T \tilde v_t^2{\rm d}t<+\infty.
\end{equation}

In the sequel, we will show the last assertion. Indeed, we will show that inequality \eqref{eq:15} fails to hold if there exists two functions $\tilde u_t,\tilde v_t: \T\To \R^+$ such that $\as$,
\begin{equation}\label{eq:16}
\bar u_t(\omega)\leq \tilde u_t\ \ \ {\rm and}\ \ \ \bar v_t(\omega)\leq \tilde v_t.
\end{equation}
Observe that for each $t\in (0,T]$,\vspace{0.1cm}
$$
\begin{array}{lll}
\Dis \left\{\omega\in \Omega: \ \bar u_t(\omega)>\tilde u_t \right\} &=&\Dis \left\{\omega\in \Omega: \int_0^t |B_s(\omega)|{\rm d}s\leq {M\over 2}\ \ {\rm and}\ \  |B_t(\omega)|>\tilde u_t \right\} \vspace{0.2cm}\\
&\supset&\Dis \left\{\omega\in \Omega: \sup_{s\in [0,t]}|B_s(\omega)|\leq {M\over 2t}\ \ {\rm and}\ \  |B_t(\omega)|>\tilde u_t \right\} \vspace{0.2cm}\\
&\supset&\Dis \left\{\omega\in \Omega: \sup_{s\in [0,t]}|B_s(\omega)|=B_t(\omega)\ \ {\rm and}\ \  \tilde u_t<|B_t(\omega)|\leq {M\over 2t} \right\}.
\end{array}
$$
It is not very hard to verify that for each $t\in (0,T]$, if $\tilde u_t<{M\over 2t}$, then the set in the last line has a positive probability and then $\mathbb{P}(\{\omega\in \Omega: \bar u_t(\omega)>\tilde u_t \})>0$. Consequently, if \eqref{eq:16} hold, then 
$$
\tilde u_t\geq {M\over 2t},\ \ {\rm d}t-a.e. \ {\rm on}\ \T
$$
and
$$
\tilde v_t\geq \sqrt{M\over 2t},\ \ {\rm d}t-a.e. \ {\rm on}\ \T,
$$
which means that 
$$
\int_0^T \tilde u_t {\rm d}t\geq \int_0^T {M\over 2t} {\rm d}t=+\infty\ \ \  {\rm and}\ \ \ \int_0^T \tilde v^2_t {\rm d}t\geq \int_0^T {M\over 2t} {\rm d}t=+\infty.
$$
The desired assertion is then proved.
\end{ex}

\section{Proof of the existence and uniqueness result}
\label{sec:5-proof of existence and uniqueness result}

In this section, we will give the proof of \cref{thm:MainResult} and \cref{cor:MainResult}. Firstly, by virtue of \cref{pro:Gronwall}, we can prove an important a priori estimate for the solutions of multidimensional BSDEs---\cref{pro:estimate}. The following assumption on the generator $g$ will be used.

\begin{enumerate}
\renewcommand{\theenumi}{(A)}
\renewcommand{\labelenumi}{\theenumi}
\item\label{A:AssumptionOfEstimation} There exist two processes
$\mu_\cdot \in L^\infty(\Omega; L^1(\T;\R_+))$ and $\lambda_\cdot \in L^\infty(\Omega; L^2(\T;\R_+))$ as well as a function $\kappa(\cdot)\in {\bf S}$ such that $\as$, for each $(y,z)\in\R^k\times\R^{k\times d}$,
   $$
    \langle y,g(\omega,t,y,z)\rangle\leq \mu_t(\omega)\kappa(|y|^2)+\lambda_t(\omega)|y||z|+f_t(\omega)|y|,
   $$
where $f_\cdot$ is an $(\F_t)$-progressively measurable nonnegative real-valued process with \vspace{0.1cm}
$$\E\left[\left(\int_0^T f_t{\rm d}t\right)^2\right]<+\infty.\vspace{0.2cm}$$
\end{enumerate}

\begin{pro}\label{pro:estimate}
Let $0<T\leq \infty$, $\xi\in\LT$, the generator $g$ satisfy assumption \ref{A:AssumptionOfEstimation} and $(y_t,z_t)_{t\in\T}$ be a solution of BSDE$(\xi,T,g)$ in the space $\s^2(0,T;\R^k)\times \M^2(0,T;\R^{k\times d})$. Then for each $t\in\T$, we have
\begin{equation}\label{eq:aprioriestimate}
 \E\left[\left.\sup\limits_{s\in [t,T]}|y_s|^2+\int_t^T |z_s|^2 {\rm d}s\right|\F_t \right]\leq C_t^1\E\left[\left.|\xi|^2+\int_t^T \mu_s{\rm d}s+\left(\int_t^T f_s{\rm d}s\right)^2 \right|\F_t\right]
\end{equation}
and
\begin{equation}\label{eq:aprioriestimatewithkappa}
 \E\left[\left.\sup\limits_{s\in [t,T]}|y_s|^2+\int_t^T |z_s|^2 {\rm d}s\right|\F_t \right]\leq C_t^2\E\left[\left.|\xi|^2+\int_t^T \mu_s\kappa\left(|y_s|^2\right){\rm d}s+\left(\int_t^T f_s{\rm d}s\right)^2 \right|\F_t\right],
\end{equation}
where
$$
C_t^1:=4c^2A^2{\rm e}^{2cA\left\|\int_t^T\left(\mu_s+\lambda^2_s\right){\rm d}s\right\|_\infty}\ \ {\rm and}\ \ C_t^2:=4c^2{\rm e}^{2c\left\|\int_t^T \lambda^2_s{\rm d}s\right\|_\infty}
$$
with $c\geq 1$ being a universal constant and $A\geq 1$ being the constant in the definition of set {\bf S}.\vspace{0.2cm}
\end{pro}

\begin{proof}
In view of assumption \ref{A:AssumptionOfEstimation}, using It\^{o} formula, the Burkholder-Davis-Gundy inequality and the basic inequality
\begin{equation}\label{eq:basicInequality}
2\alpha\beta\leq \eps\alpha^2+{1\over \eps}\beta^2\ \ {\rm for\ each}\ \alpha,\beta, \eps>0
\end{equation}
together with a standard computation we can deuce the existence of a constant $c\geq 1$ such that
$$
\begin{array}{lll}
 && \Dis \E\left[\left.\sup\limits_{s\in [t,T]}|y_s|^2+\int_t^T |z_s|^2 {\rm d}s\right|\F_t \right]\vspace{0.2cm}\\
 &\leq & \Dis c\E\left[\left.|\xi|^2+\int_t^T \left[\mu_s\kappa\left(|y_s|^2\right)+ \lambda_s^2|y_s|^2+f_s|y_s|\right]{\rm d}s\right|\F_t\right],\ \ \ t\in\T.
\end{array}
$$
Then, in view of the assumptions of $\xi, \mu_\cdot, \lambda_\cdot$, $f_\cdot$, $y_\cdot$ and $z_\cdot$ together with the fact that $\kappa(x)\leq A(1+x)$ with $A\geq 1$ for $x\geq 0$, it follows from \cref{pro:Gronwall} that for each $t\in\T$,
$$
\begin{array}{lll}
 && \Dis \E\left[\left.\sup\limits_{s\in [t,T]}|y_s|^2+\int_t^T |z_s|^2 {\rm d}s\right|\F_t \right]\vspace{0.2cm}\\
 &\leq & \Dis cA{\rm e}^{cA\big\|\int_t^T (\mu_s+\lambda_s^2){\rm d}s\big\|_\infty}\E\left[\left.|\xi|^2+\int_t^T \left[\mu_s+ f_s|y_s|\right]{\rm d}s\right|\F_t\right]
\end{array}
$$
and
$$
\begin{array}{lll}
 && \Dis \E\left[\left.\sup\limits_{s\in [t,T]}|y_s|^2+\int_t^T |z_s|^2 {\rm d}s\right|\F_t \right]\vspace{0.2cm}\\
 &\leq & \Dis c{\rm e}^{c\big\|\int_t^T \lambda_s^2{\rm d}s\big\|_\infty} \E\left[\left.|\xi|^2+\int_t^T \left[\mu_s\kappa\left(|y_s|^2\right)+f_s|y_s|\right]{\rm d}s\right|\F_t\right].\vspace{0.2cm}
\end{array}
$$
Finally, by using the basic inequality \eqref{eq:basicInequality} again, the desired inequalities \eqref{eq:aprioriestimate} and \eqref{eq:aprioriestimatewithkappa} follow immediately from the last two inequalities. The proposition is then proved.\vspace{0.2cm}
\end{proof}

\begin{rmk}\label{rmk:stoppingtime}
From the above proof it is not difficult to see that under the assumptions of \cref{pro:estimate}, for each $t\in\T$ and any $(\F_t)$-stopping times $\sigma, \tau$ satisfying $0\leq \sigma\leq \tau\leq T$, it holds that
$$
 \E\left[\left.\sup\limits_{s\in [t,T]}|\tilde y_s|^2+\int_t^T |\tilde z_s|^2 {\rm d}s\right|\F_t \right]\leq C_1\E\left[\left.|\tilde y_\tau|^2+\left\|\int_0^T \mu_s{\rm d}s\right\|_{\infty}+\left(\int_t^T \tilde f_s{\rm d}s\right)^2 \right|\F_t\right]
$$
and
$$
 \E\left[\left.\sup\limits_{s\in [t,T]}|\tilde y_s|^2+\int_t^T |\tilde z_s|^2 {\rm d}s\right|\F_t \right]\leq C_2\E\left[\left.|\tilde y_\tau|^2+\int_t^T \mu_s\kappa\left(|\tilde y_s|^2\right){\rm d}s+\left(\int_t^T \tilde f_s{\rm d}s\right)^2 \right|\F_t\right],
$$
where $\tilde y_s:={\bf 1}_{\sigma\leq s}y_{s\wedge \tau},\ \tilde z_s:={\bf 1}_{\sigma\leq s\leq \tau} z_s,\ \tilde f_s:={\bf 1}_{\sigma\leq s\leq \tau}f_s$,\vspace{0.2cm} 
$$
C_1:=4c^2A^2{\rm e}^{2cA\left\|\int_0^T\left(\mu_s+\lambda^2_s\right){\rm d}s\right\|_\infty}\ {\rm and}\ C_2:=4c^2{\rm e}^{2c\left\|\int_0^T \lambda^2_s{\rm d}s\right\|_\infty}.\vspace{-0.1cm} 
$$ 
And, from the above proof we can also observe that if there exists a constant $\gamma>0$ such that $|y_t|\leq \gamma$ for each $t\in\T$, then $(f_t)_{t\in\T}$ in assumption \ref{A:AssumptionOfEstimation} only needs to satisfy 
$$\E\left[\int_0^T f_t{\rm d}t\right]<+\infty,$$ 
and the estimates in \eqref{eq:aprioriestimate} and \eqref{eq:aprioriestimatewithkappa} still hold with $(\int_t^T f_s{\rm d}s)^2$ being replaced with $\int_t^T \gamma f_s{\rm d} s$.\vspace{0.2cm}
\end{rmk}

The following proposition can be derived from \cref{pro:Bihari} and \cref{pro:Gronwall}.

\begin{pro}\label{pro:5.3}
Let $0<T\leq +\infty$, $\beta_\cdot\in L^\infty(\Omega; L^1(\T;\R_+))$ and $\rho(\cdot)\in {\bf S}$. Assume that for each positive integer $n\geq 1$, $\hat\eta^n$ is an $\F_T$-measurable, integrable and non-negative real-valued random variable, $\hat Y^n_\cdot\in\s^2(0,T;\R^k)$ and $\hat Z^n_\cdot\in\M^2(0,T;\R^{k\times d})$. If $\lim\limits_{n\To\infty}\E[\hat\eta^n]=0$ and
\begin{equation}\label{eq:20}
\E\left[\sup\limits_{s\in [t,T]}|\hat Y^n_s|^2+\int_t^T |\hat Z^n_s|^2{\rm d}s\bigg|\F_t\right]\leq \E\left[\hat \eta^n+\int_t^T\beta_s\rho(|\hat Y^n_s|^2){\rm d}s\bigg|\F_t\right],\ \ t\in\T,
\end{equation}
then
\begin{equation}\label{eq:21}
\lim\limits_{n\To\infty}\E\left[\sup\limits_{s\in [0,T]}|\hat Y^n_s|^2+\int_0^T |\hat Z^n_s|^2{\rm d}s\right]=0.\vspace{0.2cm}
\end{equation}
\end{pro}

\begin{proof}
Note that $\rho(x)\leq A(1+x)$ for $x\geq 0$ by the definition of ${\bf S}$. In view of \eqref{eq:20}, it follows from \cref{pro:Gronwall} that for each $n\geq 1$,
\begin{equation}\label{eq:22}
\begin{array}{lll}
|\hat Y^n_t|^2&\leq & \Dis {\rm e}^{A\left\|\int_t^T\beta_s{\rm d}s\right\|_{\infty}}
\E\left[\hat \eta^n+\int_t^T A\beta_s{\rm d}s\bigg|\F_t\right]\vspace{0.2cm}\\
&\leq & \Dis {\rm e}^{A\left\|\int_0^T\beta_s{\rm d}s\right\|_{\infty}}\left(\E\left[\hat \eta^n\big|\F_t\right]+A\left\|\int_0^T\beta_s{\rm d}s\right\|_{\infty}\right),\ \ t\in \T.
\end{array}
\end{equation}
Set 
$$
\mu_t:=\overline{\lim\limits_{n\To\infty}}|\hat Y^n_t|^2,\ \ t\in\T.\vspace{0.2cm}
$$
Since $\lim\limits_{n\To\infty}\E[\hat\eta^n]=0$, it follows from \eqref{eq:22} that
\begin{equation}\label{eq:23}
0\leq \mu_t\leq A{\rm e}^{A\left\|\int_0^T\beta_s{\rm d}s\right\|_{\infty}}\left\|\int_0^T\beta_s{\rm d}s\right\|_{\infty}<+\infty,\ \ t\in\T.
\end{equation}
In view of the continuity and monotonicity of $\rho(\cdot)$, sending $n$ to infinity and using Fatou's lemma in \eqref{eq:20} yields that
$$
\mu_t\leq \E\left[\int_t^T\beta_s\rho(\mu_s){\rm d}s\bigg|\F_t\right],\ \ t\in\T.\vspace{0.2cm}
$$
Then, in view of \eqref{eq:23}, it follows from \cref{pro:Bihari} that $\mu_t=0$ for each $t\in\T$. Finally, in view of the definition of $\mu_\cdot$, Fatou's lemma and the continuity and monotonicity of $\rho(\cdot)$ together with the fact of $\rho(0)=0$, the desired conclusion \eqref{eq:21} follows from \eqref{eq:20} by letting $t=0$ and then sending $n\To \infty$. The proposition is then proved.\vspace{0.2cm}
\end{proof}

The following proposition considers a special case of \cref{thm:MainResult}.

\begin{pro}\label{pro:IndependentofZ}
Let $0<T\leq \infty$ and the generator $g$ satisfy assumptions \ref{H:ContinuousInY}--\ref{H:LipschitzInZ}. If the generator $g$ is independent of the state variable $z$, then for each $\xi\in\LT$, BSDE$(\xi,T,g)$ admits a unique solution $(y_t,z_t)_{t\in\T}$ in the space $\s^2(0,T;\R^k)\times \M^2(0,T;\R^{k\times d})$.\vspace{0.2cm}
\end{pro}

\begin{proof}
Thanks to \cref{pro:estimate} and \cref{pro:5.3} together with \cref{rmk:stoppingtime}, following closely the proof procedure of Proposition 10 in \citet{XiaoFan2017Stochastics} we can deduce the desired conclusion. The detailed  proof is omitted here.\vspace{0.2cm}
\end{proof}

Now, we can give the proof of \cref{thm:MainResult}.

\begin{proof}[Proof of \cref{thm:MainResult}]
In view of \cref{pro:estimate} and \cref{pro:Bihari}, by a similar argument to the proof of the uniqueness part of Theorem 6 in \cite{XiaoFan2017Stochastics} we can prove the uniqueness part. \vspace{0.2cm}

In the sequel, we show the existence part. We let $(y^0_\cdot,z^0_\cdot):=(0,0)$ and use Picard's iteration method. First, since $g$ satisfies assumptions \ref{H:ContinuousInY}--\ref{H:LipschitzInZ} and $\xi\in\LT$, it is not very hard to verify that for each $n\geq 1$ and $z^n_\cdot\in \M^2(0,T;\R^{k\times d})$, the generator $g(t,y,z^{n-1}_t)$ satisfies all assumptions in \cref{pro:IndependentofZ}, see page 801 in \cite{XiaoFan2017Stochastics} for details. It then follows from \cref{pro:IndependentofZ} that we can define recursively $$(y^n_t,z^n_t)_{t\in\T}\in \s^2(0,T;\R^k)\times \M^2(0,T;\R^{k\times d})$$
by the unique solution of the following BSDEs:
\begin{equation}\label{eq:24}
    y^n_t=\xi+\int_t^T g(s,y^n_s,z^{n-1}_s){\rm d}s-\int_t^T z^n_s{\rm d}B_s,
    \ \ t\in\T.\vspace{0.2cm}
\end{equation}

Next, we will show that the sequence of processes $(y^n_\cdot,z^n_\cdot)_{n\geq 1}$ is Cauchy in the whole space $\s^2(0,T;\R^k)\times \M^2(0,T;\R^{k\times d})$ by dividing the time interval $\T$ into a finite number of subintervals with stopping time ends. For each $n,i\geq 1$, set 
$$
\hat y^{n,i}_\cdot:=y^{n+i}_\cdot-y^n_\cdot\ \ {\rm and}\ \ \hat z^{n,i}_\cdot:=z^{n+i}_\cdot-z^n_\cdot
$$ 
Then, $(\hat y^{n,i}_t,\hat z^{n,i}_t)_{t\in\T}$ solves the following BSDE:
$$
\hat y^{n,i}_t=\int_t^T \hat g^{n,i}(s,\hat y^{n,i}_s){\rm d}s
    -\int_t^T \hat z^{n,i}_s{\rm d} B_s,\ \  t\in\T,
$$
where for each $y\in\R^k$, 
$$\hat g^{n,i}(t,y):=g(t,y+y^n_t,z^{n+i-1}_t)-g(t,y^n_t,z^{n-1}_t).$$ 
Furthermore, from assumptions \ref{H:WeakMonotonicityInY} and \ref{H:LipschitzInZ} it is not difficult to verify that the generator $\hat g^{n,i}$ satisfies the assumption \ref{A:AssumptionOfEstimation} with 
$$
\mu_\cdot=u_\cdot,\ \ \kappa(\cdot)=\rho(\cdot),\ \ \lambda_\cdot\equiv 0,\ \ {\rm and}\ \ f_\cdot=|\hat z^{n-1,i}_\cdot| v_\cdot.
$$ 
It then follows from \cref{pro:estimate} and \cref{rmk:stoppingtime} together with H\"older's inequality that for each $n,i\geq 1$, $t\in\T$ and any $(\F_t)$-stopping times $\sigma, \tau$ satisfying $0\leq \sigma\leq \tau\leq T$, we have
\begin{equation}\label{eq:25}
\begin{array}{lll}
&&\Dis \E\left[\left.\sup\limits_{s\in [t,T]}|\tilde y^{n,i}_{s}|^2+\int_t^T |\tilde z^{n,i}_s|^2 {\rm d}s\right|\F_t \right]\vspace{0.2cm}\\
&\leq & \Dis C\left(\E\left[\left. |\tilde y^{n,i}_{\tau}|^2+\int_0^T \tilde v^2_s{\rm d}s \int_t^T |\tilde z^{n-1,i}_s|^2{\rm d}s \right|\F_t\right]+\left\|\int_0^T u_s{\rm d}s\right\|_{\infty}\right)
\end{array}
\end{equation}
and\vspace{0.2cm}
\begin{equation}\label{eq:26}
\begin{array}{lll}
&&\Dis \E\left[\left.\sup\limits_{s\in [t,T]}|\tilde y^{n,i}_{s}|^2+\int_t^T |\tilde z^{n,i}_s|^2 {\rm d}s\right|\F_t \right]\vspace{0.2cm}\\
&\leq & \Dis C\E\left[\left.|\tilde y^{n,i}_{\tau}|^2+\int_t^T u_s\rho\left(|\tilde y^{n,i}_s|^2\right){\rm d}s+ \int_0^T \tilde v^2_s{\rm d}s \int_t^T |\tilde z^{n-1,i}_s|^2{\rm d}s \right|\F_t\right],\vspace{0.3cm}
\end{array}
\end{equation}
where $\tilde y^{n,i}_s:={\bf 1}_{\sigma\leq s}\hat y^{n,i}_{s\wedge \tau},\ \tilde z^{n,i}_s:={\bf 1}_{\sigma\leq s\leq \tau}\hat z^{n,i}_s,\ \tilde v_s:={\bf 1}_{\sigma\leq s\leq \tau}v_s$ and 
$$
C:=4c^2A^2{\rm e}^{2cA\left\|\int_0^T u_s {\rm d}s\right\|_\infty}\vspace{-0.2cm}
$$ 
with $c\geq 1$ being a universal constant and $A\geq 1$ being the constant in the definition of set {\bf S}.\vspace{0.3cm}

Now, let us fix arbitrarily a positive integer $N$ satisfying that
$$
\frac{M}{N}\leq \frac{1}{4C}
$$
with
$$
 M:=\left\|\int_0^T \left(u_s+v^2_s\right){\rm d}s\right\|_{\infty},\vspace{0.2cm}
$$
and define the following $(\F_t)$-stopping times: $T_0=0$;\vspace{0.2cm}
$$
\begin{array}{ll}
&\Dis T_1=\inf\left\{t\geq 0:\int_0^t v_s^2{\rm d}s\geq \frac{M}{N}\right\}\wedge T; \\
&\qquad\vdots \\
&\Dis T_j=\inf\left\{t\geq T_{j-1}:\int_0^t v_s^2 {\rm d}s\geq \frac{jM}{N}\right\}\wedge T; \\
&\qquad\vdots \\
&\Dis T_N=\inf\left\{t\geq T_{N-1}:\int_0^t v_s^2 {\rm d}s\geq \frac{NM}{N}\right\}\wedge T=T.\vspace{0.2cm}
\end{array}
$$
Thus, we have subdivided the time interval $[0,T]$ into a finite number of stochastic intervals $[T_{j-1},T_j],\ j=1,\cdots, N$. And, for each $j=1,\cdots, N$, we have
\begin{equation}\label{eq:27}
\int_0^T \left({\bf 1}_{{T_{j-1}}\leq s\leq {T_j}}v_s^2\right){\rm d}s\leq \frac{M}{N}\leq \frac{1}{4C},\ \ \ps\vspace{0.2cm}
\end{equation}

Furthermore, in view of \eqref{eq:27} and the fact of $\hat y^{n,i}_T=0$, letting $\sigma=T_{N-1}$ and $\tau=T_N=T$ in \eqref{eq:25} and \eqref{eq:26} yields that for each $t\in\T$,
\begin{equation}\label{eq:28}
\E\left[\left.\sup\limits_{s\in [t,T]}|\hat Y^{n,i}_s|^2+\int_t^T |\hat Z^{n,i}_s|^2 {\rm d}s\right|\F_t \right]\leq CM+{1\over 4}\E\left[\left.\int_t^T |\hat Z^{n-1,i}_s|^2{\rm d}s \right|\F_t\right]
\end{equation}
and
\begin{equation}\label{eq:29}
\begin{array}{lll}
&& \Dis \E\left[\left.\sup\limits_{s\in [t,T]}|\hat Y^{n,i}_{s}|^2+\int_t^T |\hat Z^{n,i}_s|^2 {\rm d}s\right|\F_t \right]\vspace{0.2cm}\\
&\leq & \Dis C\E\left[\left.\int_t^T u_s\rho\left(|\hat Y^{n,i}_s|^2\right){\rm d}s\right|\F_t\right]
+{1\over 4}\E\left[\left.\int_t^T |\hat Z^{n-1,i}_s|^2{\rm d}s \right|\F_t\right],
\end{array}
\end{equation}
where 
$$
\hat Y^{n,i}_s:=\hat y^{n,i}_{s\wedge T} {\bf 1}_{T_{N-1}\leq s},\ \ {\rm and}\ \ \hat Z^{n,i}_s:=\hat z^{n,i}_s {\bf 1}_{T_{N-1}\leq s\leq T}, \ s\in \T.
$$
Thus, thanks to \eqref{eq:28}, using a similar induction analysis to that in pages 802-803 in \cite{XiaoFan2017Stochastics} we can derive that for each $n\geq 1$,
\begin{equation}\label{eq:30}
\begin{array}{lll}
\Dis \sup\limits_{i\geq 1}\left(\E\left[\left.\sup\limits_{s\in [t,T]}|\hat Y^{n,i}_s|^2+\int_t^T |\hat Z^{n,i}_s|^2 {\rm d}s\right|\F_t \right]\right)&\leq &\Dis 2CM+{1\over 2}\E\left[\left.\int_0^T|z^1_s|^2{\rm d}s\right|\F_t\right]\vspace{0.1cm}\\
&=:& \Dis \overline M_t<+\infty,\ \ \ t\in\T.
\end{array}
\end{equation}
In view of \eqref{eq:30}, Fatou's lemma and the continuity and monotonicity of $\rho(\cdot)$, taking first the supremum with respect to $i$ and then the super limit with respect to $n$ on both sides of \eqref{eq:29} we can obtain that 
$$
\mu_t\leq C\E\left[\left.\int_t^T u_s\rho(\mu_s){\rm d}s\right|\F_t\right],\ \ t\in\T
$$
with
$$
\mu_t:=\overline{\lim\limits_{n\To\infty}}\sup\limits_{i\geq 1}|\hat Y^{n,i}_t|^2<+\infty.\vspace{0.2cm}
$$
Applying \cref{pro:Bihari} to the last inequality yields that 
$$
\mu_t=\lim\limits_{n\To\infty}\sup\limits_{i\geq 1}|\hat Y^{n,i}_t|^2=0,\ \ t\in\T.\vspace{0.1cm}
$$ 
Now, taking the supremum with respect to $i$ and the super limit with respect to $n$ on both sides of \eqref{eq:29} again leads to that for each $t\in\T$,
\begin{equation}\label{eq:31}
\lim\limits_{n\To\infty}\sup\limits_{i\geq 1} \E\left[\left.\sup\limits_{s\in [t,T]}|\hat y^{n,i}_{s\wedge T}{\bf 1}_{T_{N-1}\leq s}|^2+\int_t^T |\hat z^{n,i}_s{\bf 1}_{T_{N-1}\leq s\leq T}|^2 {\rm d}s\right|\F_t \right]=0.\vspace{0.3cm}
\end{equation}

Finally, in view of \eqref{eq:27} and \eqref{eq:30}, letting $\sigma=T_{N-2}$ and $\tau=T_{N-1}$ in \eqref{eq:25} and \eqref{eq:26} yields that for each $t\in\T$,\vspace{0.1cm}
\begin{equation}\label{eq:32}
\E\left[\left.\sup\limits_{s\in [t,T]}|\tilde Y^{n,i}_s|^2+\int_t^T |\tilde Z^{n,i}_s|^2 {\rm d}s\right|\F_t \right]\leq C(M+\overline M_t)+{1\over 4}\E\left[\left.\int_t^T |\tilde Z^{n-1,i}_s|^2{\rm d}s \right|\F_t\right]\vspace{0.2cm}
\end{equation}
and
\begin{equation}\label{eq:33}
\begin{array}{lll}
&&\Dis \E\left[\left.\sup\limits_{s\in [t,T]}|\tilde Y^{n,i}_{s}|^2+\int_t^T |\tilde Z^{n,i}_s|^2 {\rm d}s\right|\F_t \right]\vspace{0.2cm}\\
&\leq &\Dis C\E\left[\left.|\hat y^{n,i}_{T_{N-1}}|^2\right|\F_t\right]+C\E\left[\left.\int_t^T u_s\rho\left(|\tilde Y^{n,i}_s|^2\right){\rm d}s\right|\F_t\right]+{1\over 4}\E\left[\left.\int_t^T |\tilde Z^{n-1,i}_s|^2{\rm d}s \right|\F_t\right],
\end{array}
\end{equation}
where \vspace{0.1cm}
$$
\tilde Y^{n,i}_s:=\hat y^{n,i}_{s\wedge T_{N-1}} {\bf 1}_{T_{N-2}\leq s\leq T_{N-1}},\ \ {\rm and}\ \ \tilde Z^{n,i}_s:=\tilde z^{n,i}_s {\bf 1}_{T_{N-2}\leq s\leq T_{N-1}}, \ s\in \T.\vspace{0.2cm}
$$ 
Then, thanks to \eqref{eq:31}, by a same analysis as in the last paragraph we can use
\eqref{eq:32} and \eqref{eq:33} to deduce that for each $t\in\T$,\vspace{0.1cm}
\begin{equation}\label{eq:34}
\lim\limits_{n\To\infty}\sup\limits_{i\geq 1} \E\left[\left.\sup\limits_{s\in [t,T]}|\hat y^{n,i}_{s\wedge T_{N-1}}{\bf 1}_{T_{N-2}\leq s}|^2+\int_t^T |\hat z^{n,i}_s{\bf 1}_{T_{N-2}\leq s\leq T_{N-1}}|^2 {\rm d}s\right|\F_t \right]=0.\vspace{0.2cm}
\end{equation}
We proceed the above procedure to derive that for each $j=3,\cdots, N$ and $t\in\T$,\vspace{0.1cm}
\begin{equation}\label{eq:35}
\lim\limits_{n\To\infty}\sup\limits_{i\geq 1} \E\left[\left.\sup\limits_{s\in [t,T]}|\hat y^{n,i}_{s\wedge T_{N-j+1}}{\bf 1}_{T_{N-j}\leq s}|^2+\int_t^T |\hat z^{n,i}_s{\bf 1}_{T_{N-j}\leq s\leq T_{N-j+1}}|^2 {\rm d}s\right|\F_t \right]=0.\vspace{0.2cm}
\end{equation}
Thus, combining \eqref{eq:31}, \eqref{eq:34} and \eqref{eq:35} yields that\vspace{0.1cm}
\begin{equation}\label{eq:36}
\lim\limits_{n\To\infty}\sup\limits_{i\geq 1} \E\left[\sup\limits_{s\in [t,T]}|\hat y^{n,i}_s|^2+\int_t^T |\hat z^{n,i}_s|^2 {\rm d}s\bigg|\F_t\right]=0,\ \ t\in\T,\vspace{0.1cm}
\end{equation}
which means that $(y^n_\cdot,z^n_\cdot)_{n\geq 1}$ is a Cauchy sequence in the space $\s^2(0,T;\R^k)\times \M^2(0,T;\R^{k\times d})$. We denote the limit process by $(y_t,z_t)_{t\in\T}$ and take limit under the uniform convergence in probability in \eqref{eq:24} to see, in view of \eqref{eq:36} together with assumptions \ref{H:ContinuousInY}, \ref{H:GeneralizedGeneralGrowthInY} and \ref{H:LipschitzInZ}, that $(y_t,z_t)_{t\in\T}$ is the desired solution to BSDE$(\xi, T, g)$ in the space $\s^2(0,T;\R^k)\times \M^2(0,T;\R^{k\times d})$. The proof is then completed.\vspace{0.2cm}
\end{proof}

\begin{proof}[Proof of \cref{cor:MainResult}] Thanks to \cref{thm:MainResult}, we can use a similar argument to that in Theorem 12 of \cite{XiaoFan2017Stochastics} to obatin the desired conclusion. The details are omitted here.
\end{proof}

\vspace{0.2cm}



\setlength{\bibsep}{2pt}

\end{document}